\documentclass[]{amsart}
\usepackage{amsfonts}
\usepackage{amsmath}
\usepackage{amssymb}
\usepackage{array}
\usepackage{amsthm}
\usepackage{subfigure}
\usepackage{mathdots}
\usepackage{algorithm}
\usepackage[noend]{algpseudocode}
\usepackage{graphicx}
\usepackage{color}
\usepackage{pgf}
\usepackage{scrpage2}
\usepackage{multirow}
\usepackage{listings} 
\usepackage{tikz}
\usepackage{graphicx}
		
\newcommand{\N}{\ensuremath{\mathbb{N}}}

\renewcommand{\S}{\ensuremath{\mathbb{S}}}

\newcommand{\Z}{\ensuremath{\mathbb{Z}}}

\newcommand{\R}{\ensuremath{\mathbb{R}}}
\newcommand{\C}{\ensuremath{\mathbb{C}}}

\newcommand{\e}{\textnormal{e}}

\DeclareMathOperator*{\diag}{diag}

\DeclareMathOperator*{\Beta}{Beta}
\DeclareMathOperator*{\per}{per}
\DeclareMathOperator*{\id}{id}
\DeclareMathOperator*{\Real}{Re}
\DeclareMathOperator*{\Imag}{Im}

\newtheorem{thm}{Theorem}[section]
\newtheorem{lemma}[thm]{Lemma}
\newtheorem{remark}[thm]{Remark}
\newtheorem{definition}[thm]{Definition}
\newtheorem{example}[thm]{Example}
\newtheorem{corollary}[thm]{Corollary}
\newtheorem{proposition}[thm]{Proposition}

\numberwithin{equation}{section}
\numberwithin{table}{section}
\numberwithin{figure}{section}

\newcommand{\bend}{\hspace*{0ex} \hfill \hbox{\vrule height
    1.5ex\vbox{\hrule width 1.4ex \vskip 1.4ex\hrule  width 1.4ex}\vrule
    height 1.5ex}}

\long\def\symbolfootnote[#1]#2{\begingroup%
\def\thefootnote{\fnsymbol{footnote}}\footnote[#1]{#2}\endgroup}

\renewcommand{\mathbf}[1]{\ensuremath{\boldsymbol{#1}}}

\renewcommand{\thefootnote}{\fnsymbol{footnote}}

\allowdisplaybreaks
\title{A Randomized Multivariate Matrix Pencil Method for Superresolution Microscopy}

\date{\today}
\date{}

\author[M.~Ehler]{Martin Ehler}
\address[M.~Ehler]{University of Vienna,
Department of Mathematics,
Oskar-Morgenstern-Platz 1, 
A-1090 Vienna
}
\email{martin.ehler@univie.ac.at}

\author[S.~Kunis]{Stefan Kunis}
\address[S.~Kunis]{Osnabr\"uck University,
Institute for Mathematics, 
Albrechtstr.~28a, 
D-49076 Osnabr\"uck
}
\email{stefan.kunis@uni-osnabrueck.de}
\author[T.~Peter]{Thomas Peter}
\address[T.~Peter]{University of Vienna,
Department of Mathematics,
Oskar-Morgenstern-Platz 1, 
A-1090 Vienna
}
\email{thomas.peter@univie.ac.at}

\author[C.~Richter]{Christian Richter}
\address[C.~Richter]{Osnabr\"uck University,
Institute for Biology,
Barbarastr. 11, 
D-49076 Osnabr\"uck
}
\email{christian.richter@biologie.uni-osnabrueck.de}

\begin{document}

\begin{abstract}
The matrix pencil method is an eigenvalue based approach for the parameter identification of sparse exponential sums. 
We derive a reconstruction algorithm for multivariate exponential sums that is based on simultaneous diagonalization.
Randomization is used and quantified to reduce the simultaneous diagonalization to the eigendecomposition of a single random matrix. To verify feasibility, the algorithm is applied to synthetic and experimental fluorescence microscopy data. 
\end{abstract}

\keywords{frequency analysis,
spectral analysis,
exponential sum,
moment problem,
super-resolution}

\subjclass[2010]{65T40, 
42C15, 
30E05, 
65F30
}
\maketitle

\section{Introduction}
\label{sect:Einleitung}
Many imaging and data analysis problems in the applied sciences lead to the numerical task of parameter identification in exponential sums $\sum_{j=1}^M c_j \e^{-2\pi i\langle t_j,\cdot\rangle}$. For sparse exponential sums, i.e., for small $M$, Prony's method enables the identification of its parameters $\{t_j\}_{j=1}^M\subset \R^d$ and contributions $\{c_j\}_{j=1}^M\subset\C$ from relatively few sampling values, see e.g.~\cite{Potts:2010ko,Potts:2013vb} and references therein.

The most feasible implementations for $d=1$ are based on the eigenvalue analysis of the associated Prony matrix, see e.g.~\cite{Beinert:2017gy,PP13}. The principles of the multivariate setting have been examined in \cite{KPRO16,Kunis:by,AnCa17,Mo18}, for instance, but associated numerical schemes have not been extensively studied yet. 

The works \cite{ Sa18, DiIs17, PoTa13} describe multivariate Prony methods that are based on finding zeros of several univariate respectively multivariate polynomials.
We shall completely circumvent this algebraic geometry problem by developing a numerical scheme based on a randomized multivariate matrix pencil method. We construct matrices $S_1,\ldots,S_d$ from the sampling values, so that their simultaneous diagonalization yields the parameters $\{t_j\}_{j=1}^M$. Since $S_1,\ldots,S_d$ are not normal, standard numerical algorithms for simultaneous diagonalization are not available, cf.~\cite{Bunse-Gerstner:1993jy,Cardoso:1996ck,Golub:1996fk,Kressner:2005sp}. To circumvent this problem, we derive the joint eigenbasis from the eigendecomposition of a single matrix that is a random linear combination of $S_1,\ldots,S_d$. While \cite{AnCa17} diagonalizes $S_1$ and hopes for simple eigenvalues, the recent papers \cite[Alg.~3.1]{Mo18} and the algorithm introduced in \cite{SaUsCo17} also use the above random linear combination and argue that generically the eigenvalues are simple. While in \cite{SaUsCo17} the authors focus on analyzing the influence of pertubations on their multivariate ESPRIT-method, here in the new multivariate matrix pencil method, we describe the situation of using a random linear combination of $S_1, \dots, S_d$ in more detail and quantify the influence of the minimal separation of $\{t_j\}_{j=1}^M$ on the eigendecomposition of the random matrix. 

To check on its feasibility, our methodology is applied to analyze fluorescence microscopy images. We cast the problem of locating protein markers as a parameter identification in exponential sums. Due to its analytic roots, Prony's method enables the identification of locations at the subpixel scale, sometimes referred to as superresolution fluorescence microscopy, cf.~\cite{Studer:2012oq}. The results on experimental fluorescence images show that our scheme is numerically feasible. 

The outline is as follows: In Section \ref{sect:pre} we develop our numerical scheme. The approach of simultaneous diagonalization to identify $\{t_j\}_{j=1}^M$ is presented in Section \ref{sec:sim}. The problem of simultaneous diagonalization is reduced to the diagonalization of a single random matrix in Section \ref{sec:single}, where we examine the influence of the minimal separation of the parameters $\{t_j\}_{j=1}^M$. Our new scheme is applied to  synthetic and to experimental fluorescence microscopy data in Section \ref{sec:appl}.

\section{Reconstruction of sparse exponential sums from samples}\label{sect:pre}
Let $\{t_j\}_{j=1}^M\subset [0,1)^d$ always denote $M$ pairwise different $d$-dimensional parameters and consider the exponential sum
\begin{equation}\label{eq:fund prob samp}
f(k) = \sum_{j=1}^M c_j \e^{-2\pi \mathrm i\langle t_j,k\rangle},\quad k\in\Z^d,
\end{equation}
with nonzero coefficients $\{c_j\}_{j=1}^M\subset\C\backslash \{0\}$. Our aim is to identify the parameters $\{t_j\}_{j=1}^M$ and coefficients $\{c_j\}_{j=1}^M$ from sampling values $\{f(k)\}_{k\in I}$ with suitable $I\subset \Z^d$.

\subsection{Reconstruction by simultaneous diagonalization}\label{sec:sim}
For $n\in\N$, let $I_n:=\{0,\dots,n\}^d$ and select a fixed ordering of the elements in $I_n$. Knowledge of the sampling values of $f$ on the set difference $I:=I_{n+1}-I_n$ 
enables us to build the matrices
\begin{equation*}
T:= \left(f(k-l)\right)_{k,l\in I_n},\qquad T_{\ell}:=(f(k-l+e_{\ell}))_{k,l\in I_n}, \quad\ell=1,\ldots,d.
\end{equation*}
If $T$ has rank $M$, then we compute the reduced singular value decomposition 
\begin{equation*}
T=U \Sigma V^*,
\end{equation*}
where $\Sigma\in\R^{M\times M}$ is positive definite and $U\in\C^{N\times M}$, $V\in\C^{M\times N}$ satisfy $U^*U=V^*V=\id\in\R^{M\times M}$ with $N:=\# I_n =(n+1)^d$. Therefore, we can define the set of $M\times M$ matrices
\begin{equation}\label{eq:S def}
S_{\ell}:=U^* T_{\ell} V \Sigma^{-1},\quad \ell=1,\ldots,d.
\end{equation}
These matrices turn out to be simultaneous diagonalizable, cf.~Theorem \ref{th:22}, which shall enable us to identify the vectors $\{t_j\}_{j=1}^M$. 

In the following theorem, $K_d$ denotes an absolute constant that only depends on $d$ and is further specified in \cite{Kunis:by,KPRO16}. We also make use of 
\begin{equation*}
z_j:=\e^{-2\pi i t_j}:=(\e^{-2\pi i t_{j,1}},\ldots,\e^{-2\pi i t_{j,d}}), \quad j=1,\ldots,M,
\end{equation*} 
so that it is sufficient to reconstruct $\{z_j\}_{j=1}^M$ in order to identify $\{t_j\}_{j=1}^M$.
\begin{thm}\label{th:22}
If $n\geq  \frac{K_d}{\min_{i\neq j} \|z_i - z_j\|}$, then $T$ has rank $M$ and $S_1,\ldots,S_d$ are simultaneously diagonalizable. Furthermore, any regular matrix $W$ that simultaneously diagonalizes $S_1,\ldots,S_d$ yields a permutation $\tau$ on $\{1,\ldots,M\}$ such that 
\begin{equation*}
W^{-1} S_\ell W = \diag(\langle z_{\tau(1)},e_\ell\rangle,\ldots,\langle z_{\tau(M)},e_\ell\rangle),\quad \ell=1,\ldots,d.
\end{equation*}
\end{thm}
\begin{proof}
According to \cite{KPRO16}, $T$ always admits the factorization
\begin{equation}\label{eq:factorT}
  T= A^* D A,
\end{equation}
where $A$ is the $M\times N$ multivariate complex Vandermonde matrix
\begin{equation*}
  A=\big(z_j^k\big)_{\substack{j=1,\dots,M\\k\in I_n}},  
\end{equation*}
and $D=\diag(c_1,\ldots,c_M)$. The condition on $n$ implies that $A$ has full rank $M$, cf.~\cite{Kunis:by,KPRO16}.  Hence, $T$ has indeed rank $M$ since all $c_1,\ldots,c_M$ are nonzero.

We also deduce the factorization
\begin{equation*}
T_\ell = A^* D_\ell A,\quad \ell=1,\ldots,d,
\end{equation*}
where the diagonal matrix $D_\ell$ is given by
\begin{equation*}
D_\ell:=\diag(c_1\langle z_1,e_\ell\rangle,\ldots,c_M\langle z_M,e_\ell\rangle), \quad \ell=1,\ldots,d.
\end{equation*}
We shall now check that the specific matrix $W_0:=(AU)^*$ (which is not accessible to us) simultaneously diagonalizes $S_1,\ldots,S_d$. Indeed, by inserting the definitions, we obtain
\begin{equation*}
W_0^{-1} S_\ell W_0   = (AU)^{-*} U^* A^* D_\ell A  V  \Sigma^{-1} (AU)^*.
\end{equation*}
Note that the reduced singular value decomposition implies that both matrices, $AU$ and $AV$, are regular.
Since $\Sigma = U^*TV = U^*A^*DAV$, we deduce $\Sigma^{-1}=(AV)^{-1} D^{-1}(AU)^{-*}$, which implies 
\begin{equation*}
W_0^{-1} S_\ell W_0   =D_\ell D^{-1} = \diag(\langle z_1,e_\ell\rangle,\ldots,\langle z_{M},e_\ell\rangle),\quad \ell=1,\ldots,d,
\end{equation*}
so that $W_0$ simultaneously diagonalizes $S_1,\ldots,S_d$. Note that $W_0$ also diagonalizes any complex linear combination
\begin{equation}\label{eq:Cmu}
C_\mu:=\sum_{\ell=1}^d \overline{\mu}_\ell S_\ell,\quad \mu\in\C^d.
\end{equation}
Because of
\begin{equation*}
	W_0^{-1} C_\mu W_0 = \diag\left(\sum_{\ell = 1}^d \bar \mu_\ell\langle z_1,e_\ell \rangle, \ldots,\sum_{\ell = 1}^d \bar \mu_\ell\langle z_1,e_\ell \rangle \right),
\end{equation*}
the eigenvalues $\lambda_1(\mu),\ldots,\lambda_M(\mu)$ of $C_\mu$ are 
\begin{equation*}
\lambda_j(\mu)=\langle z_j,\mu\rangle
\end{equation*}
with the ordering induced by $W_0$. 
Since $\{t_j\}_{j=1}^M$ are pairwise different, so are $\{z_j\}_{j=1}^M$, and, hence, there is $\tilde{\mu}\in \S_\C^{d-1}=\{x\in\mathbb{C}^d:\|x\|=1\}$ such that $\langle z_i - z_j, \tilde \mu \rangle \neq 0$ for all $i\neq j$ and thus $\{\lambda_j(\tilde{\mu})\}_{j=1}^M$ are pairwise different. In other words, all eigenspaces of $C_{\tilde{\mu}}$ are $1$-dimensional. 

Any matrix $W=(w_1,\ldots,w_M)$ that simultaneously diagonalizes $S_1,\ldots,S_d$ also diagonalizes $C_{\tilde{\mu}}$. 
Thus, there is a permutation $\tau$ such that $w_{\tau(i)}$ spans the same space as the $i$-th column of $W_0$, which concludes the proof.
\end{proof}
According to Theorem \ref{th:22}, the diagonalization of $S_\ell$ encodes the $\ell$-th entry of a permutation of the  vectors $\{z_j\}_{j=1}^M$. We require simultaneous diagonalization to ensure that these entries are associated to the same permutation across all $\ell=1,\ldots,d$.

In general, the matrices $S_1,\ldots,S_d$ are not normal. Therefore, the numerical task of simultaneous diagonalization is difficult and many simultaneous diagonalization algorithms in the literature are not suitable, cf.~\cite{Bunse-Gerstner:1993jy,Cardoso:1996ck,Golub:1996fk,Kressner:2005sp}. We attempt to circumvent such issues by using $C_\mu$ from \eqref{eq:Cmu}, which shall enable us to restrict our diagonalization efforts to a single matrix:
\begin{corollary}\label{th:single}
If $\mu\in\C^d$ is such that $\lambda_1(\mu),\ldots,\lambda_M(\mu)$ are pairwise different, then any matrix $W$ that diagonalizes $C_\mu$ also simultaneously diagonalizes $S_1,\ldots,S_d$. 
\end{corollary}
\begin{proof}
The matrices $C_\mu, S_1,\ldots,S_d$ are simultaneously diagonalizable. The same arguments as in the proof of Theorem \ref{th:22} imply the assertion. 
\end{proof}
According to Corollary \ref{th:single} we aim to find $\mu\in\C^d$ such that $\lambda_1(\mu),\ldots,\lambda_M(\mu)$ are pairwise different. 
For a nonzero vector $z\in\C^d$, let $z^\perp$ denote the $d-1$-dimensional linear subspace of $\C^d$ orthogonal to $z$. 
The proof of Theorem \ref{th:22} reveals that
\begin{equation}\label{eq:set}
\big\{\mu \in\C^d : \lambda_1(\mu),\ldots,\lambda_M(\mu) \text{ are pairwise different} \big\} = \C^d\setminus \bigcup_{i\neq j}(z_i-z_j)^\perp
\end{equation}
Hence, this set is the entire $\C^d$ except for at most $\binom{M}{2}$ many $(d-1)$-dimensional subspaces.

\begin{example}
Let $d=2$, $M=5$, and choose $t_1, \dots, t_5 \in [0,1)^2$ randomly. We construct $S_1, S_2 \in\mathbb C^{5\times 5}$ by \eqref{eq:S def}. Thus, we choose $\mu = (\mu_1, \mu_2)^\top \in\mathbb{S}_\C^1$ and construct $C_\mu = \mu_1S_1 + \mu_2 S_2$. According to \eqref{eq:set} we expect $\binom{5}{2} = 10$ great circles on $\mathbb S_\C^1$, with the property that choosing a $\mu$ from one of those great circles results in a $C_\mu$, that has at least one eigenspace of dimension larger than one. 
For $\xi \in\mathbb C$, with $\|\xi\|=1$ we get $C_{\mu \xi} = \xi\left(\mu_1S_1 + \mu_2 S_2\right)$. This shows that the multiplication of $C_\mu$ by a global phase $\xi$ does not change the pairwise differences of the eigenvalues of $C_\mu$ and therefore we can use Hopf fibration, to identify great circles on $\mathbb S_\C^1$ with a single point on $\mathbb S^2$, for visualization. Indeed we can observe that the minimal distance of any two eigenvalues of $C_\mu$ is nonzero on $\mathbb{S}^2$ except for $10$ points, see Figure \ref{fig:MuAbhaengigkeit}(a). Note, that we only see $8$ of those $10$ points in \ref{fig:MuAbhaengigkeit}(a), the other $2$ are on the back side of the sphere.

%
For visual illustration of the expected great circles, we now switch to the real case and choose $d=3$, $M=5$, 
and restrict $\mu$ to the real sphere $\mathbb S^2$. In  Figure \ref{fig:MuAbhaengigkeit}(b) we see $10$ great circles on $\mathbb S^2$, for which $C_\mu$ has eigenspaces of dimension larger than one. Observe that away from those great circles, the minimal distance of any two eigenvalues of $C_\mu$ rapidly increases. 
\begin{figure}
		\subfigure[$S_1, S_2 \in\mathbb C^{5 \times 5}$, $\mu \in \mathbb S_\C^1$ ]{
	  \includegraphics[width=0.45\textwidth]{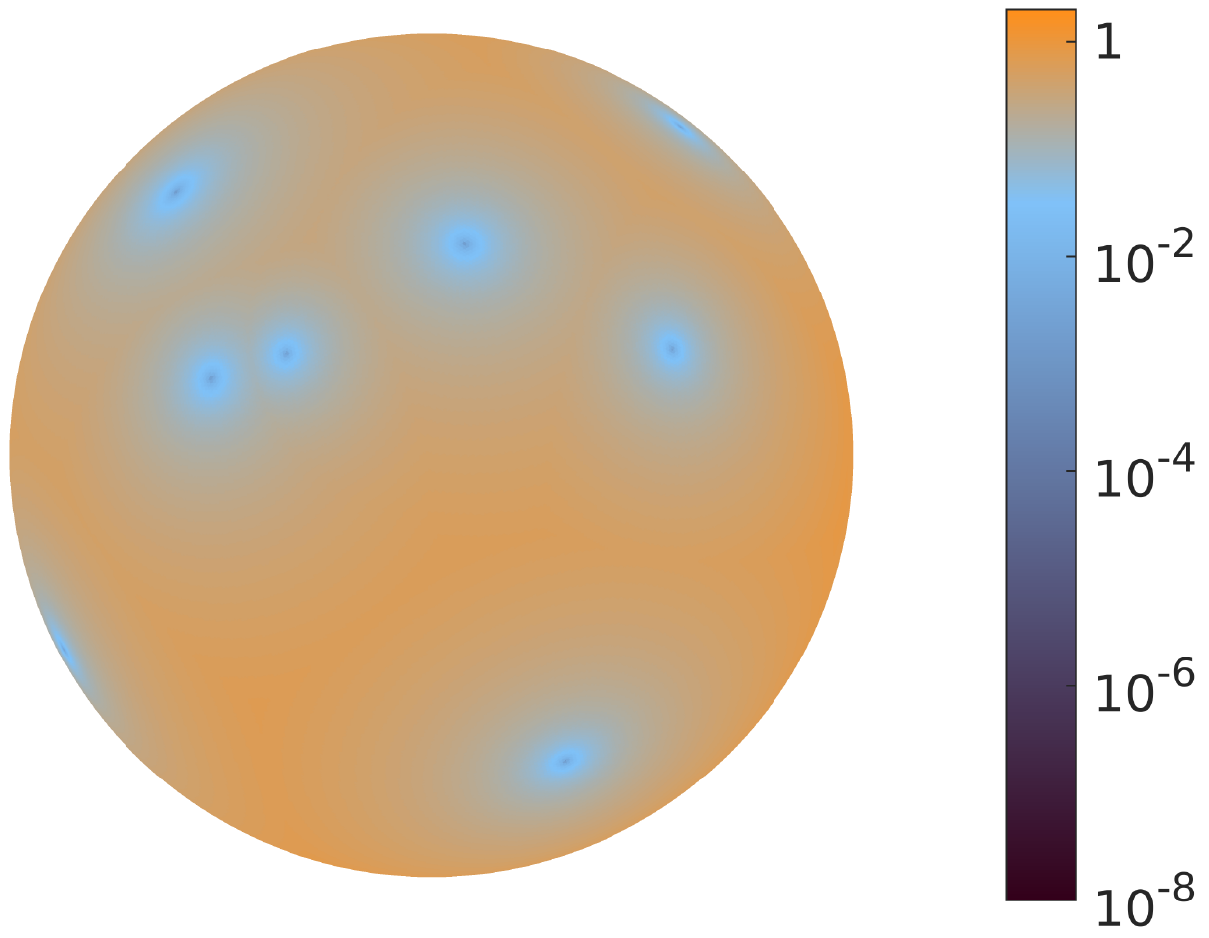}
	}
	\subfigure[$d=3$, $M=5$, and $\mu \in \mathbb S^2$ ]{
	  \includegraphics[width=0.45\textwidth]{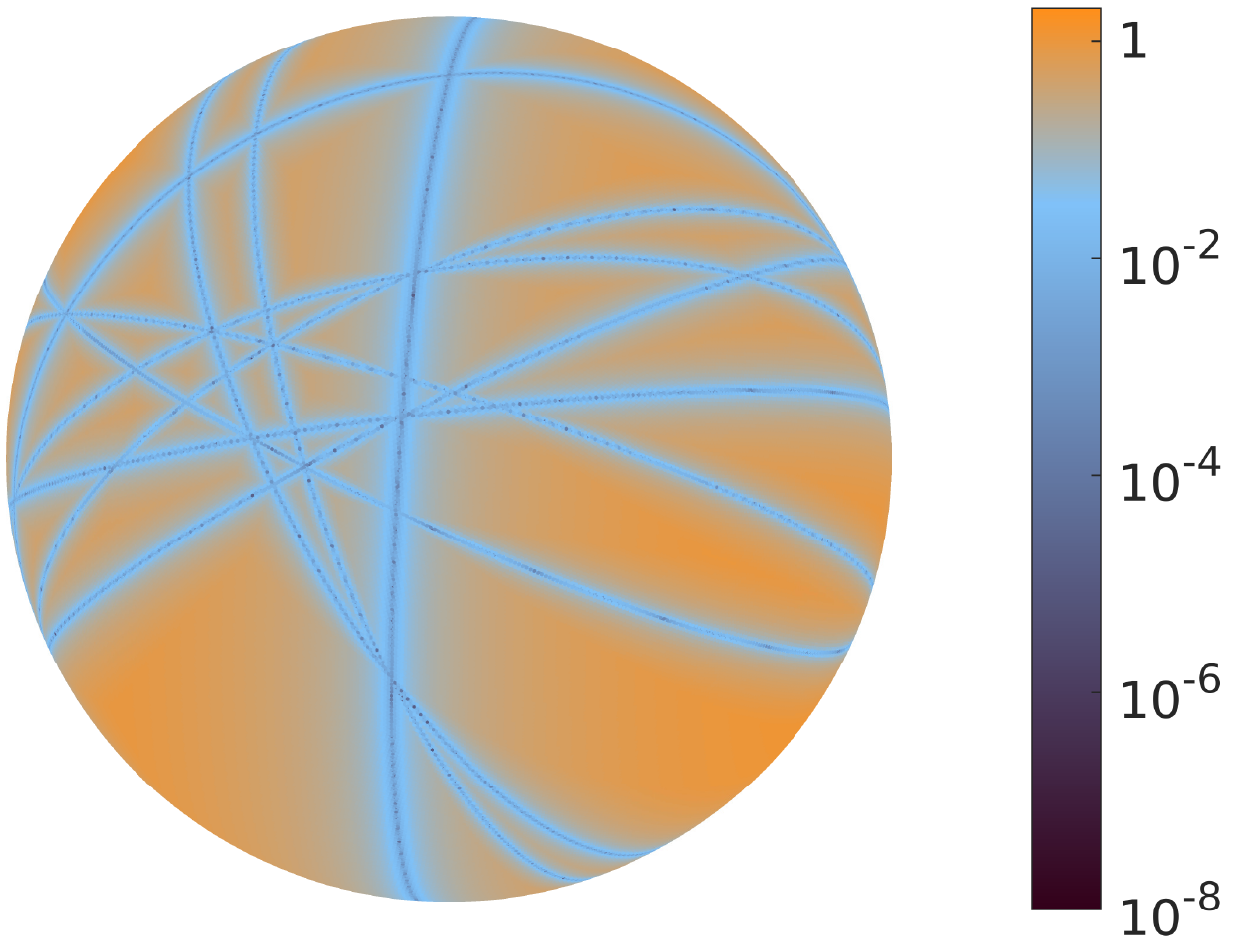}
	}
	\caption{Visualization of the smallest distance of any two eigenvalues of $C_\mu$.}
	\label{fig:MuAbhaengigkeit}
	\end{figure}
\end{example}

\begin{remark}
Our approach to simultaneous diagonalization of $S_1,\ldots,S_d$ suggested in Corollary \ref{th:single} requires our present setting, in which $\{z_j\}_{j=1}^M$ are pairwise different. It does not apply to the problem of simultaneous diagonalization in general. 
\end{remark}

\subsection{Simultaneous diagonalization by random linear combinations}\label{sec:single}
The present section is dedicated to quantify the difference $\lambda_i(\mu)-\lambda_j(\mu)$ in relation to the difference $z_i-z_j$. If $\mu\in\S_\C^{d-1}$ is a random vector, distributed according to the unitarily invariant probability measure on $\S_\C^{d-1}$, then 
\begin{equation*}
\mathbb{E}|\lambda_i(\mu)-\lambda_j(\mu)| = \frac{1}{\sqrt{d}}\|z_i-z_j\|.
\end{equation*}
The following result provides a more quantitative analysis:
\begin{thm}\label{th:stab oder so}
Let $i\neq j$ be fixed and suppose $\epsilon\in[0,1]$. If $\mu\in\S_\C^{d-1}$ is a random vector, distributed according to the unitarily invariant probability measure on $\S_\C^{d-1}$, then the probability that 
\begin{equation}\label{eq:mu satifsies}
|\lambda_i(\mu)-\lambda_j(\mu)| < \epsilon \|z_i-z_j\|
\end{equation}
holds is at most $2\sqrt{\frac{d}{\pi}}\epsilon$.
\end{thm}
Theorem \ref{th:stab oder so} immediately implies that the probability that any of the inequalities 
\begin{equation}\label{eq:mu satifsies}
|\lambda_i(\mu)-\lambda_j(\mu)|\geq \epsilon \|z_i-z_j\|, \quad \forall i\neq j,
\end{equation}
is violated is at most $\binom{M}{2} 2\sqrt{\frac{d}{\pi}}\epsilon$. 
In other words, if we select about $M^2$ many independent $\mu$, then the probability that \eqref{eq:mu satifsies} fails is at most of the order $\epsilon$. 
\begin{proof}[Proof of Theorem \ref{th:stab oder so}]
The complex sphere $\S_\C^{d-1}$ admits the standard identification with the real sphere $\S^{2d-1}$ by $x\mapsto \Big(\begin{smallmatrix}\Real(x)\\ \Imag(x)
\end{smallmatrix}\Big)
$, 
and $\Big(\begin{smallmatrix}\Real(\mu)\\ \Imag(\mu)
\end{smallmatrix}\Big)$ is distributed according to the orthogonal invariant probability measure on $\S^{2d-1}$, the latter being the standard normalized surface measure. 

Let $y:=\frac{z_i-z_j}{\|z_i-z_j\|}\in \mathbb{S}_\C^{d-1}$, so that $|\lambda_i(\mu)-\lambda_j(\mu)|/\|z_i-z_j\|=|\langle y,\mu\rangle|$. Since 
\begin{equation}\label{eq:Real Im}
\Big|
\left\langle  \big(\begin{smallmatrix} \Real(y)\\ \Imag(y)
\end{smallmatrix}\big), \big(\begin{smallmatrix} \Real(\mu)\\ \Imag(\mu)
\end{smallmatrix}\big)\right\rangle\Big| = |\Real\big(\langle y,\mu\rangle\big)| \leq |\langle y,\mu\rangle|,
\end{equation}
we obtain an upper bound by simply considering
\begin{equation}\label{eq:dfrt}
\Big|
\left\langle  \big(\begin{smallmatrix} \Real(y)\\ \Imag(y)
\end{smallmatrix}\big), \big(\begin{smallmatrix} \Real(\mu)\\ \Imag(\mu)
\end{smallmatrix}\big)\right\rangle\Big| \leq \epsilon.
\end{equation}
Due to the orthogonal invariance of the surface measure on $\S^{2d-1}$, the distribution of the left-hand-side in \eqref{eq:dfrt} does not  depend on the special choice of $y\in\mathbb{S}_\C^{d-1}$, so that we can simply assume that $\Big(\begin{smallmatrix}\Real(y)\\ \Imag(y)
\end{smallmatrix}\Big)$ is the north pole. The inequality \eqref{eq:dfrt} reduces to $-\epsilon\leq \Real(\mu_1)\leq \epsilon$, hence, describes the complement of two opposing spherical caps in $\S^{2d-1}$. This ``equatorial band'' has measure 
\begin{equation*}
1-\mathcal{I}_{[1-\epsilon^2]}(d-\frac{1}{2},\frac{1}{2}) = \mathcal{I}_{[\epsilon^2]}(\frac{1}{2},d-\frac{1}{2}),
\end{equation*}
see, for instance, \cite{Li:2011id}, where $\mathcal{I}_{[x]}(a,b)$ is the cumulative distribution function of the Beta distribution, i.e., 
\begin{equation*}
\mathcal{I}_{[x]}(a,b) = \frac{\int_0^x t^{a-1}(1-t)^{b-1}dt }{\Beta(a,b)},\qquad \Beta(a,b) = \frac{\Gamma(a)\Gamma(b)}{\Gamma(a+b)}.
\end{equation*}
%
%
%
%
For $d=1$, we observe 
\begin{equation*}
\mathcal{I}_{[\epsilon^2]}(1/2,1/2) = \frac{2\arcsin(\epsilon)}{\pi} \leq \frac{2}{\sqrt{\pi}}\epsilon.
\end{equation*}
Suppose now $d\geq 2$ and define 
\begin{equation*}
f(x):=2\sqrt{x} -  \mathcal{I}_{[x]}(1/2,d-1/2) \Beta(1/2,d-1/2).
\end{equation*}
A short calculation yields that its derivative satisfies
\begin{equation*}
f'(x)=\frac{1-(1-x)^{d-3/2}}{\sqrt{x}}\geq 0,\quad x\in[0,1].
\end{equation*}
Since $f(0)=0$, we obtain 
\begin{equation}
\mathcal{I}_{[\epsilon^2]}(\frac{1}{2},d-\frac{1}{2}) \leq \frac{2\epsilon}{\Beta(1/2,d-1/2)},\quad\epsilon\in[0,1].
\end{equation}
The observation $1/\Beta(1/2,d-1/2)\leq \sqrt{d/\pi}$ concludes the proof. 
\end{proof}
\begin{remark}
A short calculation leads to 
\begin{equation*}
\mathcal{I}_{[\epsilon^2]}(\frac{1}{2},d-\frac{1}{2})  =
 \frac{2}{\pi} \Big[  \arcsin(\epsilon) + \epsilon \sum_{k=2}^{d} \frac{4^{k-2}(k-2)!^2 }{(2k-3)(2k-4)!} (1-\epsilon^2)^{k-3/2} \Big].
 \end{equation*}
One then deduces directly that, for fixed $d$ and small $\epsilon$, the term $\mathcal{I}_{[\epsilon^2]}(\frac{1}{2},d-\frac{1}{2}) $ is of the order $\epsilon$.
\end{remark}

%
%


Theorem \ref{th:22}, Corollary \ref{th:single}, and Theorem \ref{th:stab oder so} enable us to determine $z_{\tau(1)},\ldots,z_{\tau(M)}$. The actual parameters $t_{\tau(j)}$ are computed as the principal values of $\log(z_{\tau(j)})$. The coefficients $c_{\tau(1)},\ldots,c_{\tau(M)}$ can be determined by solving the linear system $T=A^*DA$ for $D=\diag(c_{\tau(1)},\ldots,c_{\tau(M)})$ by the least squares method. We have summarized these steps in Algorithm \ref{alg_1}.
\begin{algorithm}
    \caption{Prony's method using the multivariate matrix pencil approach}\label{alg_1}
    \begin{algorithmic}[1]
\State \textbf{input} $f(k)$, $k\in I$.
\State Compute the reduced singular value decomposition of $T$.
\State Build the matrices $S_1,\ldots,S_d$.
\State Choose random $\mu\in\S_\C^{d-1}$ and compute a matrix $W$ that diagonalizes $C_\mu$.
\State Use $W$ to simultaneously diagonalize $S_1,\ldots,S_d$ and reconstruct $z_{\tau(1)},\ldots,z_{\tau(M)}$.
\State Compute $t_{\tau(j)}$ as the principal value of $\log(z_{\tau(j)})$, $j=1,\ldots,M$.
\State Solve $\mathrm{argmin}_{c}\,\,\|A^* c- f\|_2$ to recover $c_{\tau(1)},\ldots,c_{\tau(M)}$.
\State \textbf{return} $t_{\tau(1)},\ldots,t_{\tau(M)}$ and $c_{\tau(1)},\ldots,c_{\tau(M)}$.
    \end{algorithmic}
\end{algorithm}


\section{Application in superresolution microscopy}\label{sec:appl}
\subsection{Mathematical model}
In fluorescence microscopy one puts a fluorescence marker on proteins and stimulates them with a laser. 
In accordance with the fluorescent microscope's resolution limits, proteins are modeled as point sources, cf.~\cite{Studer:2012oq}, so that the probe is considered a tempered distribution 
\begin{equation}\label{eq:dira}
G = \sum_{j=1}^M c_j \delta_{t_j},
\end{equation}
on $\R^d$, where $\{t_j\}_{j=1}^M\subset [0,1)^d$ is associated to the protein locations and $\delta_{t_j}$ denotes the Dirac delta function with center $t_j$. Let $\mathcal{F}$ denote the Fourier transform on the space of tempered distributions on $\R^d$. Then $\mathcal{F}(G)$ is an exponential sum 
\begin{equation}\label{eq:ft sum}
\mathcal{F}(G)=\sum_{j=1}^M c_j \e^{-2\pi i\langle t_j,\cdot\rangle}.
\end{equation}
The actual measurements $g$ are the convolution of $G$ with some smooth and sufficiently fast decaying function $\varphi$, 
\begin{equation*}
g=G*\varphi = \sum_{j=1}^M c_j \varphi(\cdot-t_j).
\end{equation*}
Usually, $\varphi$ is modeled as a Gaussian with known parameters determined by the camera system. 

In order to determine the locations $\{t_j\}_{j=1}^M$ and the contributions $\{c_j\}_{j=1}^M$, suppose we have access to the Fourier transform of the measurements,
\begin{equation*}
\mathcal{F}(g) = \mathcal{F}(G) \mathcal{F}(\varphi).
\end{equation*}
Since $\varphi$ is known, let us also assume that we have access to $\mathcal{F}(\varphi)$. If $\varphi$ is a Gaussian, for instance, we know $\mathcal{F}(\varphi)$ analytically. We now look for some sampling set $I\subset \Z^d$, where $\mathcal{F}(\varphi)$ does not vanish, and are able to determine the right-hand-side of 
\begin{equation}\label{eq:form ert}
\mathcal{F}(G)(k)=\mathcal{F}(g)(k) / \mathcal{F}(\varphi)(k), \quad k\in I. 
\end{equation}
Combining \eqref{eq:ft sum} with \eqref{eq:form ert} leads to the sampling problem \eqref{eq:fund prob samp} discussed in the previous sections, i.e., 
\begin{equation}\label{eq:eq finale}
\sum_{j=1}^M c_j \e^{-2\pi i \langle t_j,k\rangle } = f(k),\qquad k\in I,
\end{equation}
with $f(k):=\mathcal{F}(g)(k) / \mathcal{F}(\varphi)(k)$. The parameters $\{t_j\}_{j=1}^M$ and $\{c_j\}_{j=1}^M$ can now be determined by Algorithm \ref{alg_1} in principle. Note that the above derivations in this section have also been used in \cite{PePoTa11} in combination with the univariate Prony's method.

In practice though, we are not able to numerically compute the Fourier transform of $g$ directly, so that the right-hand-side of \eqref{eq:eq finale} is not readily available. Aiming at the application of the discrete Fourier transform (DFT), we recognize that sufficient decay of $\varphi$ implies $g\in L^1(\R^d)$, so that its periodization 
\begin{equation*}
g_{\per} : = \sum_{l\in\Z^d} g(\cdot+l)
\end{equation*}
converges pointwise almost everywhere towards a function $g_{\per}\in L^1(\mathbb{T}^d)$, where $\mathbb{T}^d\simeq [0,1)^d$ is the $d$-dimensional torus. Let $\hat{g}_{\per}(k)$ denote the $k$-th Fourier coefficient of $g_{\per}$. The Poisson formula yields
\begin{equation*}
\mathcal{F}(g)(k) = \hat{g}_{\per}(k),\quad k\in I.
\end{equation*}
Thus, \eqref{eq:eq finale} can be evaluated by first computing the periodization $g_{\per}$, so that its Fourier coefficients yield
\begin{equation}\label{eq:rhs final}
\sum_{j=1}^M c_j \e^{-2\pi \mathrm i \langle t_j,k\rangle } = \hat{g}_{\per}(k) / \mathcal{F}(\varphi)(k) ,\qquad k\in I.
\end{equation}
Numerically, the DFT enables the approximation of the Fourier coefficients $\hat{g}_{\per}(k)$, $k\in I$, from samples of $g_{\per}$. 
%
%
%

It should be mentioned that all numerical experiments were realized in Python on an Intel~i7, 8GByte, 3GHz, macOS 10.12.

\subsection{Numerical results on synthetic data}
In our numerical experiments, we shall apply an implementation of the DFT to compute the discrete Fourier transform of samples of $g_{\per}$. The sampling rate of $g$ and hence $g_{\per}$ is determined by the pixel resolution. For both, synthetic and experimental fluorescence microscopy data, we choose $\varphi(\cdot) = \mathrm e^{-b \|\cdot\|^2}$ with adjusted parameter $b$ derived from the camera system. Therefore, the values $\mathcal{F}(\varphi)$ are even available in analytic form.  

Our analysis is first used on synthetic data in Figure \ref{fig:1} with 
	\begin{align*}
	t_1 &= \left(\tfrac{2}{5}, \tfrac{2}{5}\right), &c_1 &= 1,& b&=150,\\
	t_2 &= \left(\tfrac{2}{5}, \tfrac{3}{5}\right), & c_2 &= 1, \\
	t_3 &= \left(\tfrac{3}{5}, \tfrac{2}{5}\right), & c_3& = 1.
\end{align*} 
The measurements $g$ are first exact and in a second experiment corrupted by additive Gaussian noise with a signal to noise ratio of $\mathrm{SNR} = 2.554$, cf.~Figure \ref{fig:1}. For our computations we choose, if not stated otherwise,  $n=4$, so that $I=\{-4,\ldots,5\}^2$ and $T$ is an $N\times N$ Toeplitz matrix with $N = 25$. These matrix dimensions show that our methodology is numerically feasible. By examining significant drops in the singular values of $T$, we determine $M$ being $3$ for the synthetic data. The reconstructed locations $\tilde{t}_1,\tilde{t}_2,\tilde{t}_3$ satisfy $\|t_j- \tilde{t}_j\|\leq 1.88\cdot 10^{-3}$, for $i=1,2,3$, in the noisy regime, and coincide with the correct locations up to machine precision in the noise-free regime, see Figure \ref{fig:1}. It is important to note that our approach does not require the parameters $\{t_j\}_{j=1}^M$ to lie on the pixel grid. The pixel grid is only used to approximate $\hat{g}_{\per}(k)$, $k\in I$, by the DCT to determine the right-hand-side in \eqref{eq:rhs final}. 
\begin{figure}
\subfigure[Blue stars indicate the three identified locations within noiseless synthetic data.]{
  \includegraphics[width=.45\textwidth]{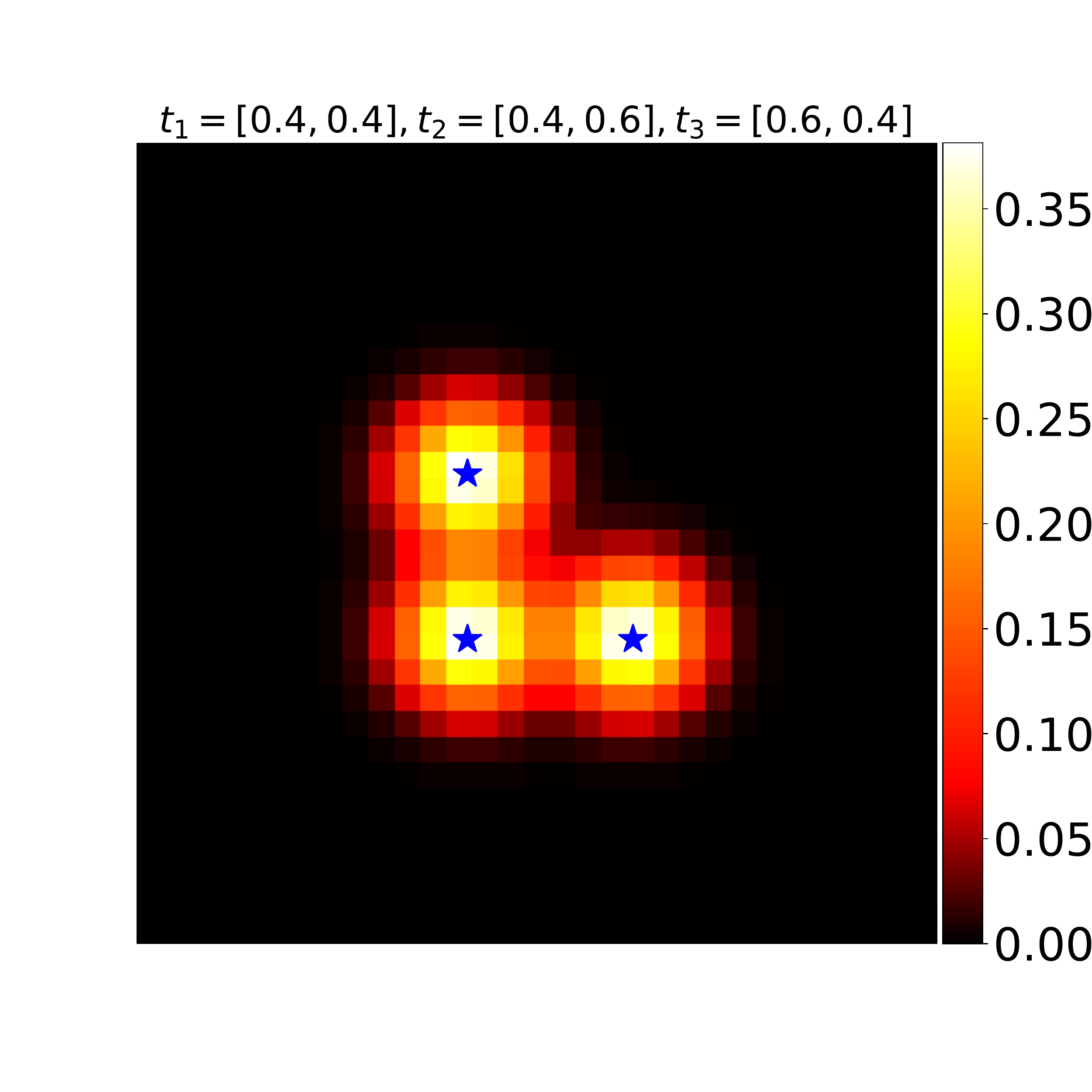}
  \label{fig:EchteDaten}}\hfill
\subfigure[Good location identification within synthetic data corrupted by additive Gaussian noise with $\mathrm{SNR}=2.554$.]{
  \includegraphics[width=.45\textwidth]{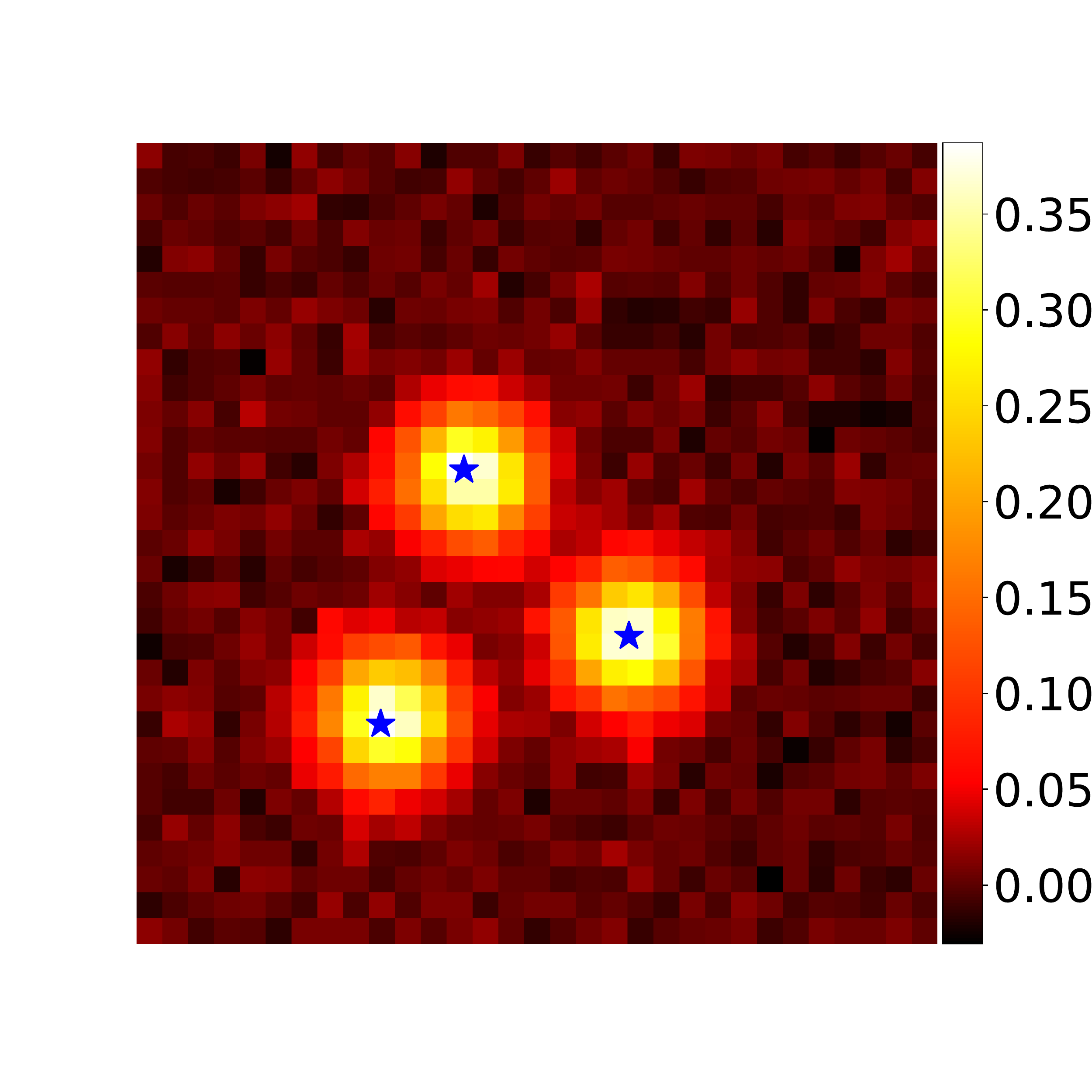}
  \label{fig:ModellAusDaten}
}
\caption{In noiseless synthetic data and in the presence of additive Gaussian noise in spatial domain, our proposed algorithm manages to find the locations $t_1, t_2, t_3$ with reasonable accuracy.}
\label{fig:1}
\end{figure}
Indeed, the locations that we compute do not lie on the pixel grid, so we are identifying locations on the subpixel level. This is an important advantage we gain by making our computations in the Fourier domain. Figure \ref{fig:SubpixelNeed} shows the difference between true locations $t_1 = 0.44, t_2 = 0.56$ of two one dimensional Gaussians, compared to the local maxima of their sum. For illustration purpose we use a one dimensional scenario in Figure \ref{fig:SubpixelNeed}. Even though this effect is negligible when $\|t_1 - t_2\|_2 \gg 0$, it would entail miscalculations when the positions $t_1, t_2$ of two proteins are close to each other. Consider a movie, where each frame is a picture as in Figure \ref{fig:ModellAusDaten} and the found locations $t_j$ are used to compute movement speeds of each protein. Then one would falsely compute an accelerated attraction and a longer contact phase of two approaching proteins if this effect is not considered.
\begin{figure}
\includegraphics[width=.75\textwidth]{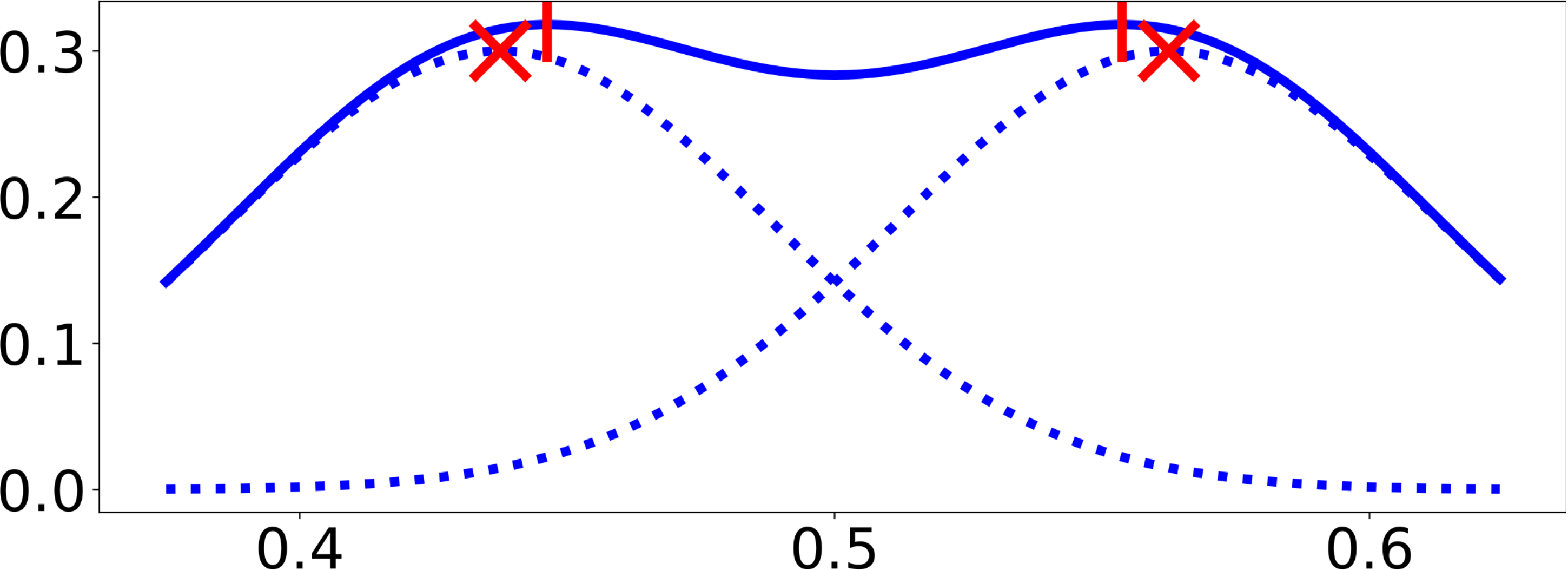}
\caption{The red crosses show the true location of $t_1= 0.44, t_2 = 0.56$ of two one-dimensional Gaussians, each depicted as a dotted line. The red bars however, show the local maxima of the sum of these gaussians and this sum is shown in a continuous line.}
\label{fig:SubpixelNeed}
\end{figure}

To illustrate potential numerical issues when the measurements are corrupted by noise, i.e., when $\tilde{g}:=g+\varepsilon$ is measured in place of $g$, we show in Figure \ref{fig:NoiseVsNoNoise} the real-parts of $\hat{g}_{\per}(k)$, $\hat{\tilde{g}}_{\per}(k)$, approximated by the DFT and $\mathcal{F}(\varphi)(k)=\hat{\varphi}_{\per}(k)$, as well as the respective ratios on a line $k_1=0$ and $k_2=-15,\ldots,15$. Even though we are dealing with images of the size $31\times 31$ pixels, the frequency data of the noisy ratio $\hat{\tilde{g}}_{\per}(k)/\hat{\varphi}_{\per}(k)$ seems only reliable close to the center. While $\hat{\varphi}_{\per}(k)$ decays with growing $k$, the noise keeps $\hat{\tilde{g}}_{\per}(k)$ from decaying, so that the ratio becomes unreasonably large. Therefore, we must restrict $n$ depending on the noise level, and $n=4$ seems to work in our synthetic data with fixed $\mathrm{SNR}$ as well as in our fluorescence microscopy data. Figure \ref{fig:WeitTeilen} shows the ratios $\hat{g}_{\per}(k)/\hat{\varphi}(k)$ for $k\in \{-4,\ldots,5\}^2$. 
\begin{figure}
\centering
\includegraphics[width=0.6\textwidth]{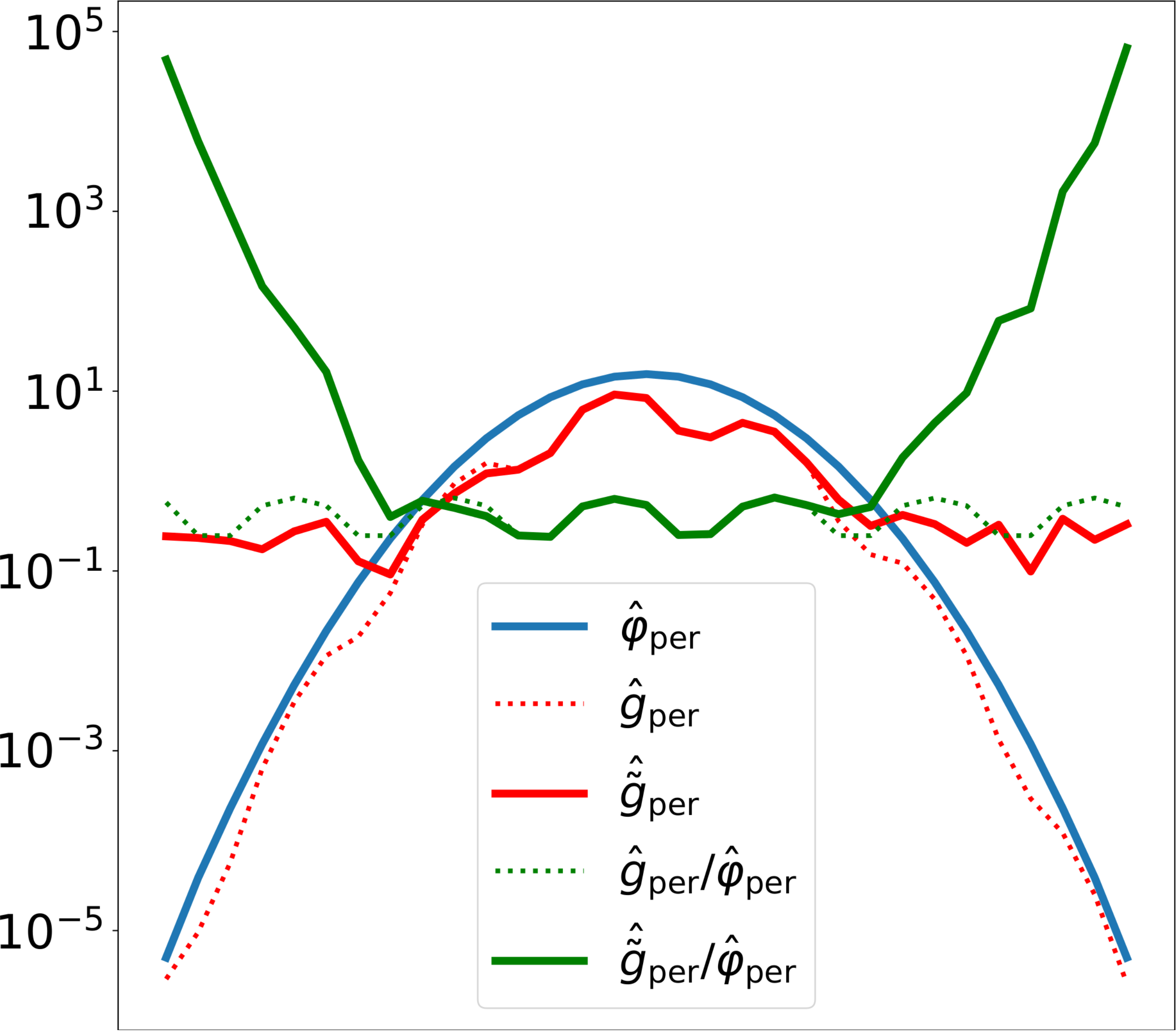}
\caption{The horizontal axis corresponds to $k_1=0$ and $k_2=-15,\ldots,15$. 
The decay of the Fourier coefficients $\hat{\tilde{g}}_{\per}$ stagnates in the presence of noise, so that the ratio $\hat{\tilde{g}}_{\per}(k)/\hat{\varphi}_{\per}(k)$ is unbounded away from the center.}
\label{fig:NoiseVsNoNoise}
\end{figure}
\begin{figure}
\subfigure[Real part]{
  \includegraphics[width=0.45\textwidth]{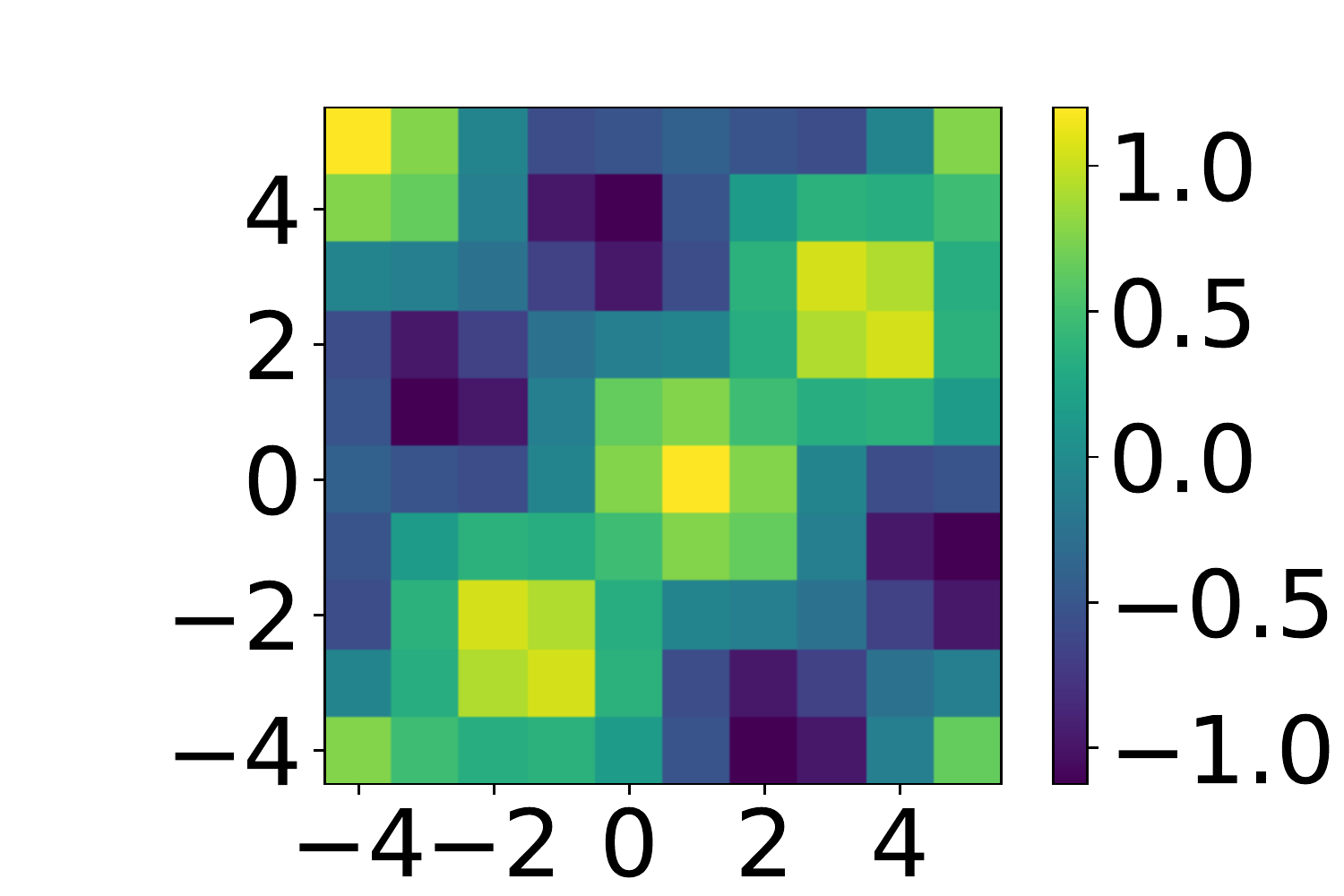}
}
\subfigure[Imaginary part]{
  \includegraphics[width=0.45\textwidth]{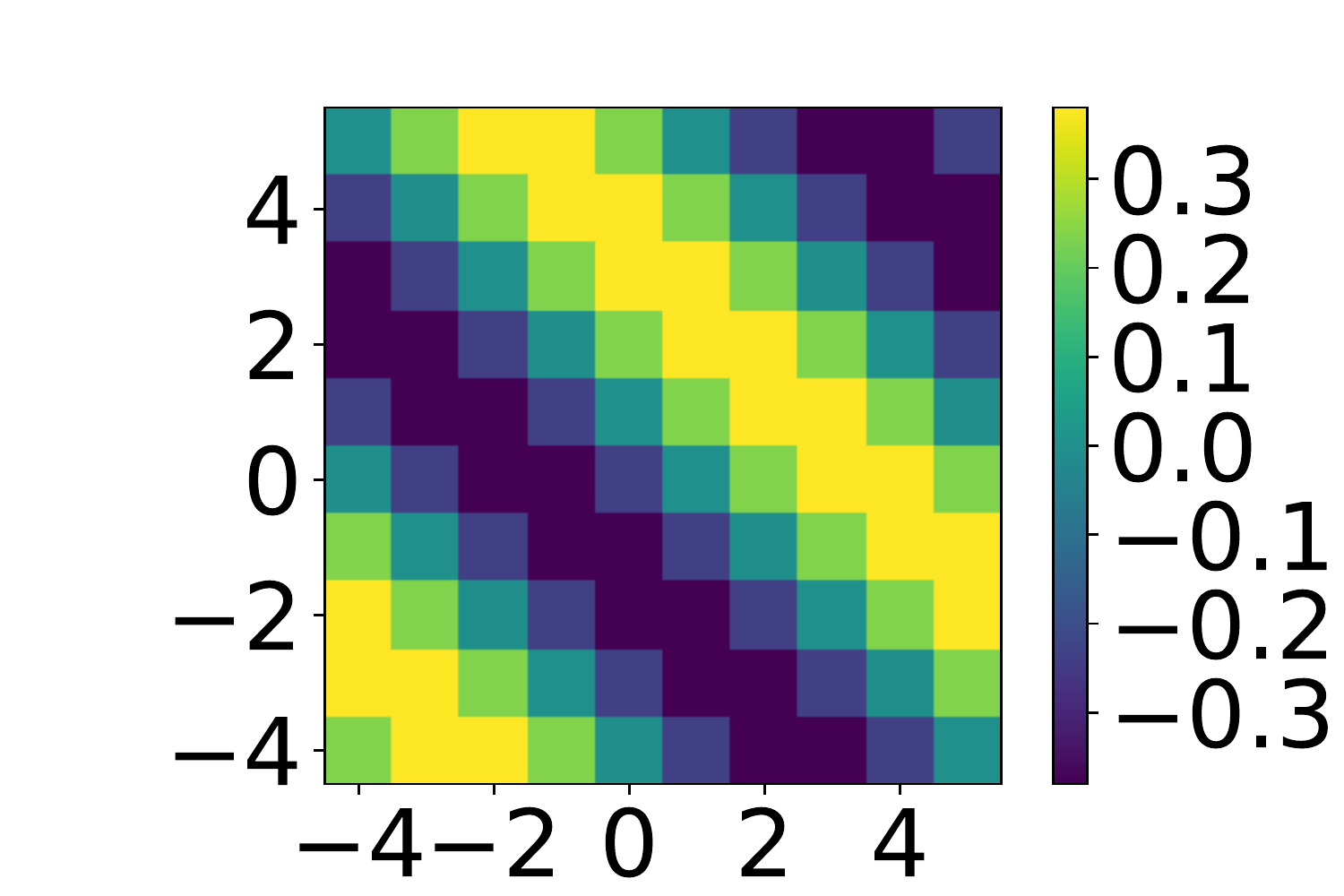}
}
\caption{$\hat{g}_{\per}(k)/\hat{\varphi}_{\per}(k)$ on $k\in \{-4,\ldots,5\}^2$.}
\label{fig:WeitTeilen}
\end{figure}

Theorem \ref{th:22} requires $n$ to be larger if the minimal separation distance 
\begin{equation*}
q:=\min_{j\neq i}\|z_j-z_i\|
\end{equation*}
becomes smaller. In Figure \ref{fig:EinflussVonQ} we illustrate this relation by two examples with noisy synthetic data, one for $q_1 = 0.283$ and the other for $q_2 = 0.057$. For $n=1$ and $n=4$, the locations can still be recovered reasonably well for $q_1$. In the case $q_2$, the choice $n=1$ fails to recover the locations that are close to each other but $n=4$ is successful. 
\begin{figure}
\subfigure[$\min_{i\neq j}\|z_i-z_j\|=0.283$: locations are recovered with error margins $\leq 7.1\cdot 10^{-3}$ and $\leq 2.8\cdot 10^{-3}$ for $n=1$ and $n=4$, respectively.]{
  \includegraphics[width=0.45\textwidth]{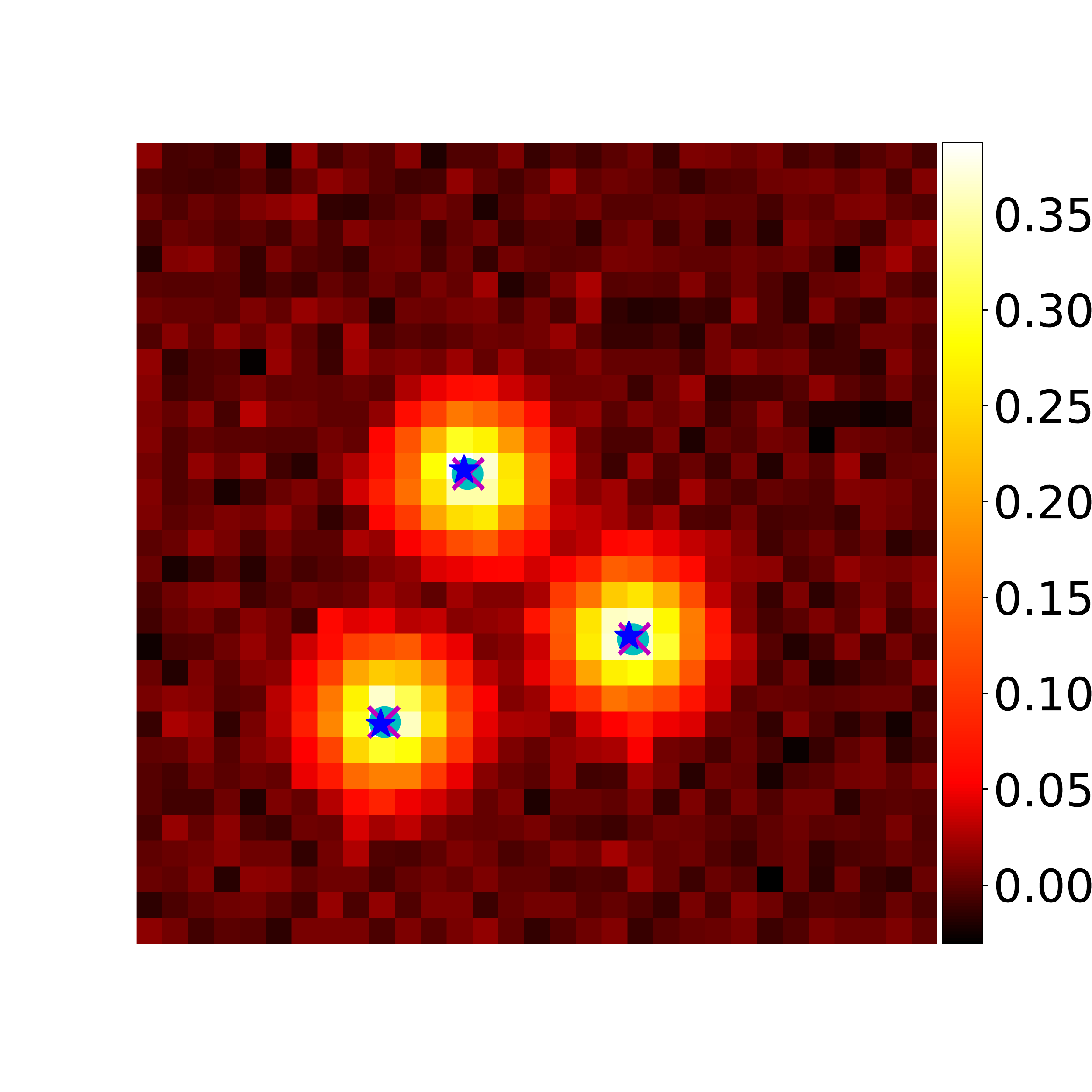}
  }
  \hfill
\subfigure[$\min_{i\neq j}\|z_i-z_j\|= 0.057$: $n=1$ fails. Locations are correctly recovered for $n=4$ with error $\leq 1\cdot 10^{-2}$.]{
  \includegraphics[width=0.45\textwidth]{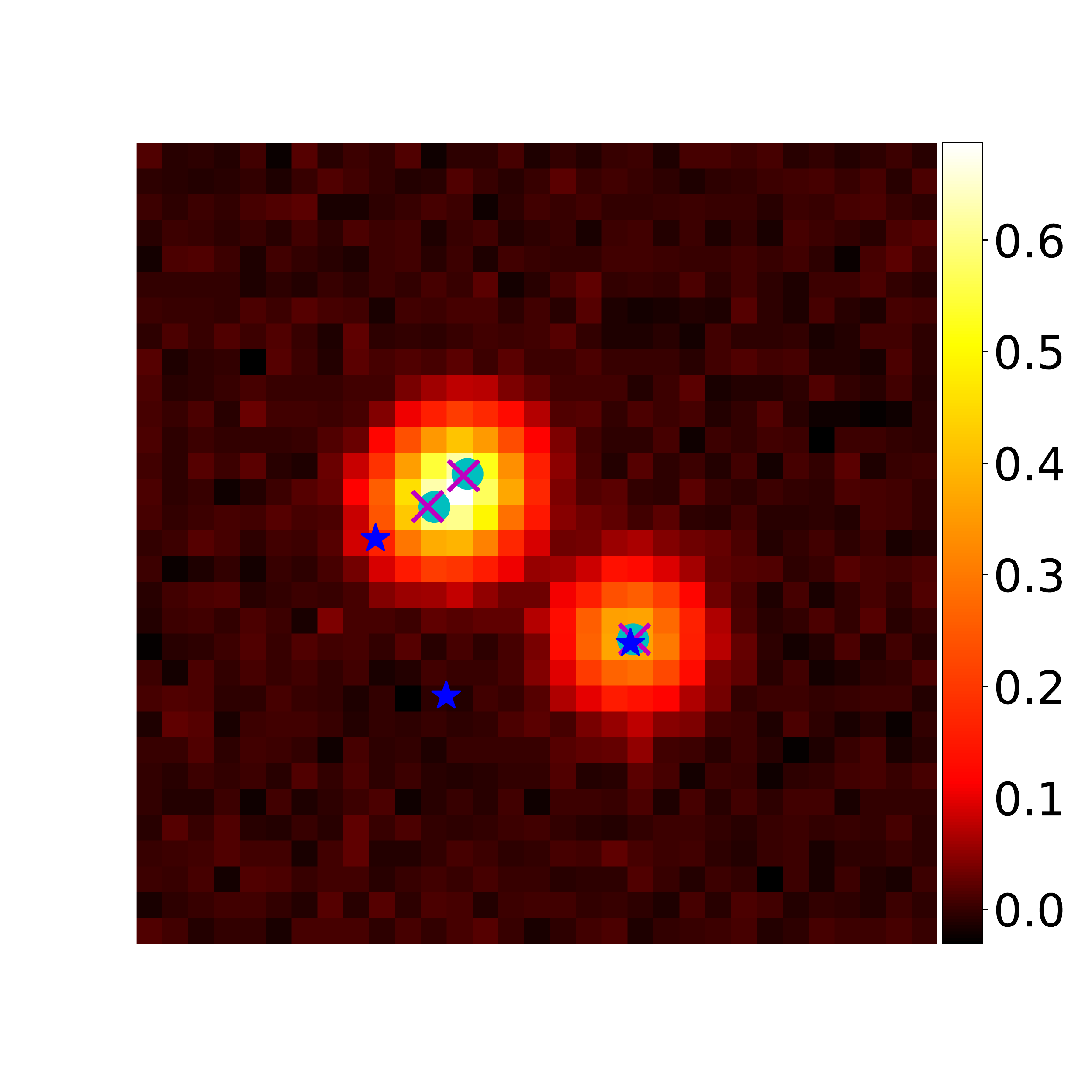}
}
\caption{Noisy synthetic data with $\mathrm{SNR}=2.554$. The light blue circles show the true locations $t_1, t_2, t_3$. The blue stars show the reconstruction with $n=1$, the magenta crosses show the reconstruction with $n=4$. In accordance with the ``spirit'' of the requirements on $n$ in Theorem \ref{th:22}, well-separated true locations allow for small $n$. If locations are not well-separated, then $n=1$ fails but the choice $n=4$ enables reconstruction.}
\label{fig:EinflussVonQ}
\end{figure}

\subsection{Numerical results on fluorescence microscopy data}
The cell-surface receptor IFNAR2 (type I interferon beta-subunit) of living cells was labelled with biofunctionalized quantum dots (QD605, Cat. No. Q21501MP, Invitrogen \cite{YoWiRiBeLiPi13}). These nanoparticles are small in size (hydrodynamic radius of 15-21 nm) but show an extraordinary high fluorescence signal. Single-molecule imaging was done on an inverted TIRF (total internal reflection fluorescence) microscope (Olympus IX71) with a scientific grade digital camera (Hamamatsu ORCA Flash 4.0). After optical magnification (150xTIRF objective UAPO; NA, 1.45; Olympus) and pixel-binning the final pixel size in the image plane was calculated to be 87 nm. To achieve a high signal-to-noise ratio the signal integration time was set to 32 ms.

The decay of the singular values of $T$ with $n=4$ for the experimental fluorescence microscopy data in Figure \ref{fig:real 1}(a) suggest $M=8$. This yields $C_\mu,S_1,S_2 \in\mathbb C^{8\times 8}$ and our algorithm finds the parameters $t_j, c_j$, $j=1,\ldots,8$, in less than a millisecond. Note in Figure \ref{fig:real 1}(b) that our algorithm, somewhat surprisingly, successfully identifies proteins at the boundary of the image, even though one would expect artifacts due to periodization issues. However, those identified translations close to the boundary are not very reliable and will need a post- or pre-processing step in a more elaborate analysis in practice. 

\begin{figure}
\subfigure[]{
  \includegraphics[width=0.45\textwidth]{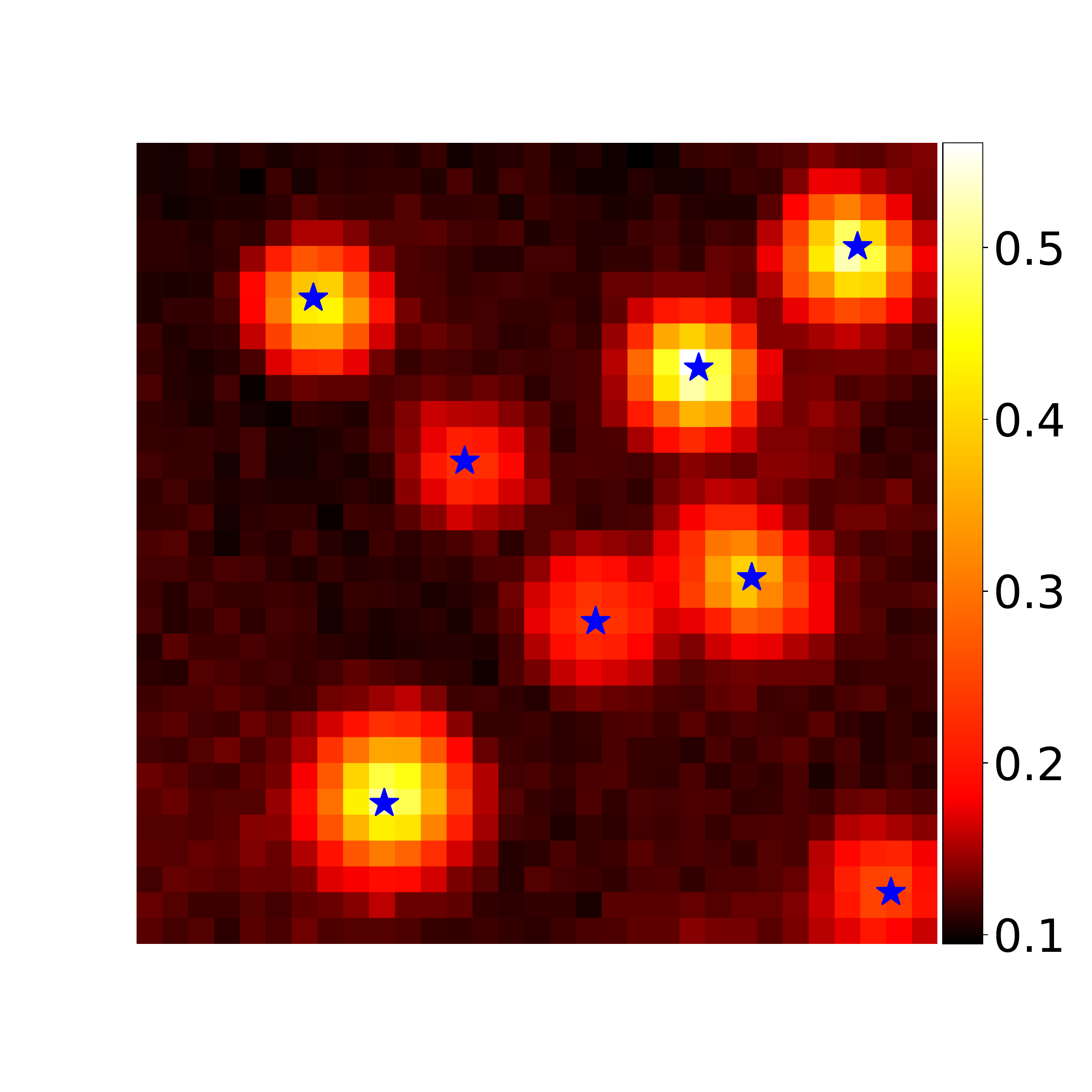}
  }
  \hfill
\subfigure[]{
  \includegraphics[width=0.45\textwidth]{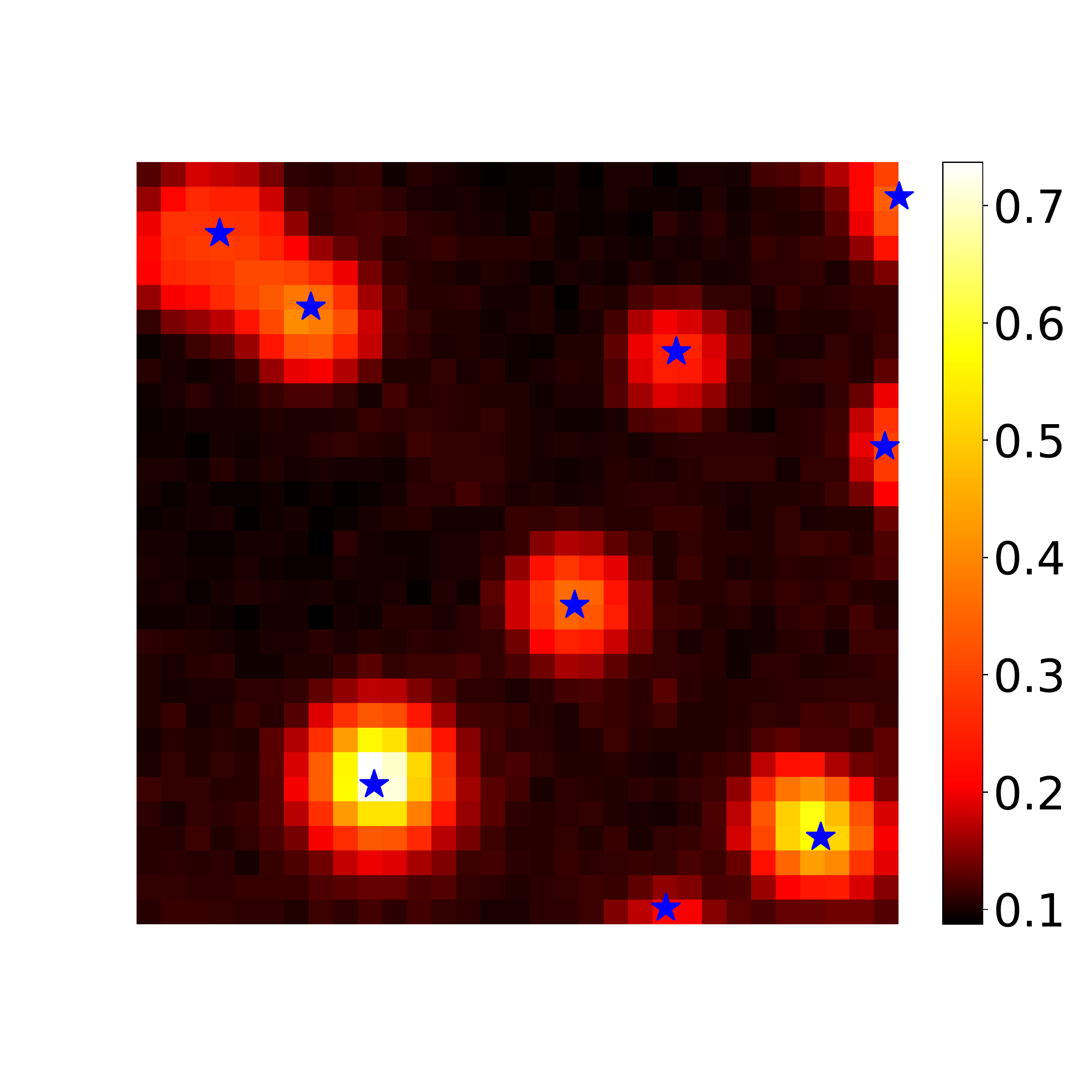}
}
\caption{Experimental data with blue stars marking identified locations.}
\label{fig:real 1}
\end{figure}

\section*{Conclusion}
We proposed an algorithm that finds multivariate frequencies out of structured samples of a finite sum of multivariate exponentials. Our proposed algorithm is a multivariate generalization of a matrix pencil method and is based on simultaneous diagonalization of a pencil of non-normal matrices. We also studied a method to simultaneously diagonalize the occurring non-normal matrices by analyzing random linear combinations. Randomness was also quantified in relation to the minimal separation of the exponential parameters. We successfully tested our algorithm on experimental data from fluorescence microscopy.

\section*{Acknowledgements}
The authors have been partially funded by WWTF through project VRG12-009, by DAAD through P.R.I.M.E.~57338904, by FWF project P30148, and by DFG-SFB944.

\providecommand{\bysame}{\leavevmode\hbox to3em{\hrulefill}\thinspace}
\providecommand{\MR}{\relax\ifhmode\unskip\space\fi MR }
\providecommand{\MRhref}[2]{%
  \href{http://www.ams.org/mathscinet-getitem?mr=#1}{#2}
}
\providecommand{\href}[2]{#2}

\end{document}